\documentclass[ECP]{ejpecp} 
\pdfoutput=1

\usepackage{latexsym,amsfonts,amsmath,amssymb,amsthm,url,graphics,psfrag,amscd,stmaryrd, mathrsfs}

\SHORTTITLE{Contact process on one-dimensional long range percolation} 

\TITLE{Contact process on one-dimensional long range percolation} 

\AUTHORS{Van Hao Can\footnote{Aix Marseille Universit\'e, CNRS, Centrale Marseille, I2M, UMR 7373, 13453 Marseille, France. \BEMAIL{cvhao89@gmail.com}}}


\AMSSUBJ{82C22; 60K35; 05C80} 

\SUBMITTED{August 3, 2015}
\ACCEPTED{December 17, 2015}

\VOLUME{20}
\YEAR{2015}
\PAPERNUM{93}
\DOI{v20-4461}

\newcommand{\E}{\mathbb{E}}
\newcommand{\Z}{\mathbb{Z}}

\newcommand{\pp}{\mathbb{P}}

\newcommand{\kC}{\mathcal{C}}
\newcommand{\sC}{\mathscr{C}}

\newcommand{\kE}{\mathcal{E}}

\newcommand{\pn}{\mathbb{P} }

\newcommand{\en}{\mathbb{E} }

\def\l{\ell}
\newcommand{\lin}{\left[\kern-0.15em\left[}
\newcommand{\rin} {\right]\kern-0.15em\right]}
\newcommand{\linf}{[\kern-0.15em [}
\newcommand{\rinf} {]\kern-0.15em ]}
\newcommand{\ilin}{\left]\kern-0.15em\left]}
\newcommand{\irin} {\right[\kern-0.15em\right[}

\ABSTRACT{Recently, by introducing  the notion of cumulatively merged partition, M\'enard and Singh provide  in \cite{MS} a sufficient condition on graphs ensuring that the critical value of   the contact process is positive.  In this note, we show  that    the one-dimensional long range percolation with high exponent satisfies their condition and thus the contact process exhibits a non-trivial phase transition.}

\begin{document}

\section{Introduction}
In this paper, we study the contact process on $G_s$, the one-dimensional long range percolation graph with exponent $s>1$, defined as follows:  independently for any $i$ and $j$ in $\mathbb{Z}$ there is an edge connecting them with probability $|i-j|^{-s}$. In particular, $G$ contains $\Z$ so it is connected.

\vspace{0.2 cm}
On the other hand, the contact process was introduced in an article of T. E. Harris \cite{H} and is defined as follows: given a locally finite graph $G=(V,E)$ and $\lambda >0$, the contact process on $G$ with infection rate $\lambda$ is a Markov process $(\xi_t)_{t\geq 0}$ on $\{0,1\}^V$. Vertices of $V$ (also called sites) are regarded as individuals which are either infected (state $1$) or healthy (state $0$). By considering $\xi_t$ as a subset of $V$ via $\xi_t \equiv \{v: \xi_t(v)=1\},$ the transition rates are given by
\begin{align*}
\xi_t \rightarrow \xi_t \setminus \{v\} & \textrm{ for $v \in \xi_t$ at rate $1,$ and } \\
\xi_t \rightarrow \xi_t \cup \{v\} & \textrm{ for $v \not \in \xi_t$ at rate }  \lambda \, \textrm{deg}_{\xi_t}(v),
\end{align*}
where $\textrm{deg}_{\xi_t}(v)$ denotes the number of infected neighbors of $v$ at time $t$. Given  $A \subset V$, we denote by $(\xi_t^A)_{t \geq 0}$ the contact process with initial configuration $A$ and if $A=\{v\}$ we simply write $(\xi^v_t)$. 

Since the contact process is monotone in $\lambda$,  we can define the critical value 
\begin{align*}
 \lambda_c(G)= \inf \{\lambda: \pp(\xi ^v_t \neq \varnothing \,\forall t) >0\}.
\end{align*} 
This definition does not depend on the choice of $v$ if $G$ is connected. If $G$ has bounded degree, then there exists a non-trivial sub-critical phase, i.e. $\lambda_c >0$, as the contact process is stochastically dominated by a continuous time branching random walk with reproduction rate $\lambda$. Thus for  integer lattices and regular trees, the critical value is  positive. The behavior of the contact process on these graphs was extensively investigated, see for instance \cite{L, P, S}.

In contrast, there is a little knowledge about the sub-critical phase on  unbounded degree graphs. For  Galton-Watson trees,  Pemantle proved in \cite{PS} that if the reproduction law $B$ asymptotically satisfies that $\pp(B \geq x) \geq \exp(-x^{1- \varepsilon})$, for some $\varepsilon>0$, then $\lambda_c =0$. Recently, in \cite{MS}, by introducing the notion of cumulatively merged partition (abbr. CMP) (see Section 2.2), the authors provided  a sufficient condition on graphs ensuring that $\lambda_c >0$. As an application, they show that the contact process on  random geometric graphs and  Delaunay triangulations  exhibits a non-trivial phase transition. 

\vspace{0.2cm}
The long range percolation graph was first introduced in \cite{Sc, ZPL}.  Then it gained interest in some contexts such as the graph distance, diameter, random walk, see  \cite{ CS2} for a list of reference. The long range percolation is locally finite if and only if $s>1$, so  we only consider the contact process on such graphs.  Moreover, it follows from the ergodicity of  $G_s$ that there is a non negative  constant $ \lambda_c(s)$, such that
\begin{align} \label{egd}
\lambda_c(G_s) = \lambda_c(s) \textrm{   for almost all graphs } G_s.
\end{align}
It is clear that the sequence of graphs $(G_s)$ is stochastically decreasing in $s$ in the sense that $G_{s_1}$ can be coupled as a subgraph of $G_{s_2}$ if $s_1 \geq s_2$. Therefore $\lambda_c(s_1) \geq \lambda_c(s_2)$. Hence, we can define 
\begin{align} 
s_c = \inf \{s: \lambda_c(s)>0\}.
\end{align}    
We will apply the method in \cite{MS} to show that $s_c < + \infty$. Here is our main result. 
\begin{theorem} \label{the}
We have
\begin{align*}
s_c \leq 102.
\end{align*}
\end{theorem}
There is a phase transition in the structure of the long range percolation. If $s<2$,  the graph $G_s$ exhibits the small-world phenomenon. More precisely, the distance between $x$ and $y$ is of order $(\log |x-y|)^{\varkappa + o(1)}$ with $\varkappa = \varkappa(s)>1$, with probability tending to 1 as $|x-y| \rightarrow \infty$, see for instance \cite{B2}. In contrast, if $s >2$,  the graph somehow looks like $\mathbb{Z}$ (see Section \ref{stg}) and the distance now is of order $|x-y|$, see \cite{BB}.  On the other hand, as mentioned above, we know that $\lambda_c(\mathbb{Z}) >0$. Hence, we conjecture that 
\begin{align*}
s_c \leq 2.
\end{align*}
The results in \cite{MS} can be slightly improved and thus we could  get a better bound on $s_c$, but it would still be far from the critical value $2$.

\vspace{0,2 cm}
\noindent The paper is organized as follows. In Section 2, we first describe the structure of the graph and show that $G_s$ can be seen as the gluing of i.i.d. finite subgraphs. Then  we recall  the definitions and  results of  \cite{MS} on the CMP. By studying the moment of the total weight of a subgraph,  we are able to apply the results from \cite{MS} and  prove our main theorem. 
\section{Proof of Theorem \ref{the}} 
\subsection{Structure of the graph} \label{stg}
We fix $s>2$. For any $k \in \mathbb{Z}$, we say that $k$ is a {\it cut-point} if there is no edge $(i,j)$ with $i<k$ and $j >k$. 
\begin{lemma} \label{lgd}
The following statements hold.
\begin{itemize}
\item[(i)] For all $k\in \mathbb{Z}$
$$\pp(k \textrm{ is a cut-point}) = \pp(0 \textrm{ is a cut-point}) >0.$$
As a consequence, almost surely there exist infinitely many cut-points. \\
\item[(ii)] The subgraphs induced in the intervals between  consecutive cut-points are i.i.d. In particular, the distances between consecutive cut-points form a sequence of i.i.d. random variables.
\end{itemize}
\end{lemma}
\begin{proof}
We first prove (i). Observe that 
\begin{eqnarray*}
\pp(k \textrm{ is a cut-point}) &= & \pp(0 \textrm{ is a cut-point}) \notag\\
&= & \prod \limits_{i<0<j} \left(1- |i-j|^{-s} \right) \notag\\
& \geq & \exp \left( -  2 \sum \limits_{i<0<j} |i-j|^{-s}  \right) \notag \\
& \geq & e^{2/(2-s)},
\end{eqnarray*}
where we used that $1-x \geq \exp(-2x)$ for $0 \leq x \leq 1/2$ and 
\begin{eqnarray*}
\sum \limits_{i<0<j} |i-j|^{-s} = \sum \limits_{i,j \geq 1} (i+j)^{-s}  & \leq & \frac{1}{s-1}\sum \limits_{i \geq 1} i^{1-s} \\
& \leq & \frac{1}{s-1} \left( 1+ \frac{1}{s-2} \right) = \frac{1}{s-2},
\end{eqnarray*}
using series integral comparison.

Then the ergodic theorem implies that there are infinitely many cut-points a.s.

Part (ii) is immediate, since there are no edges between different intervals between consecutive cut-points.
\end{proof}
We now study  some properties of the distance between two consecutive cut-points.  
\begin{proposition} \label{pgd} 
Let $D$ be the distance between two consecutive cut-points. Then there exists a sequence of integer-valued random variables $(\varepsilon_i)_{i\geq 0}$ with $\varepsilon_0=1$, such that 
\begin{itemize}
\item[(i)] $D= \sum \limits_{i=0}^T \varepsilon_i$  with $T= \inf \{i \geq 1: \varepsilon_i =0\}$, \\
\item[(ii)] $T$ is stochastically dominated by a  geometric random variable with mean $ e^{2/(2-s)}$,\\
\item[(iii)] for all $i, \ell \geq 1$
\begin{align*}
\pp(\varepsilon_i > \ell \mid T \geq i ) \leq   \ell^{2-s}/(s-2).
\end{align*}
\end{itemize}
\end{proposition}
\begin{proof} To simplify notation, we assume that $0$ is a cut-point.  Set $X_{-1}=0$ and  $X_0=1$, then  we define  for $i\geq 1$
\begin{align*}
X_i & = \max \{k: \exists \, X_{i-2} \leq j \leq X_{i-1}-1,  \,\, j \sim k\}, \\
\varepsilon_i &= X_i -X_{i-1}.
\end{align*}
Then $\varepsilon_i \geq 0$ and  we define
 \[T= \inf \{i \geq 1: X_i = X_{i-1}\} = \inf \{i \geq 1: \varepsilon_i =0 \}.\]
We have $X_i=X_{i-1}$ for all $i\geq T$, or equivalently $\varepsilon_i =0$ for all $i\geq T$. 

Note that $X_T$ is the closest cut-point on the right of $0$, so it has the same law as $D$, by definition.  Moreover
 \begin{align} \label{ect}
X_T = \sum \limits_{i=0}^T \varepsilon_i,
\end{align} 
which implies (i). Observe that for $i\geq 1$ we have $\{T \geq i\}=\{X_{i-2}<X_{i-1}\}$ and
\begin{eqnarray*}
\pp(T=i \mid T \geq i )& = & \pp (X_i = X_{i-1} \mid  X_{i-2} < X_{i-1}) \\
& = & \pp ( \nexists \,  X_{i-2} \leq j < X_{i-1}< k: \, j \sim k \mid  X_{i-2} < X_{i-1}) \\ 
& \geq & \prod  \limits_{j<0<k}\left(1- |j-k|^{-s} \right)\\
& \geq & e^{2/(2-s)}.
\end{eqnarray*}
This implies (ii). For (iii), we note that for $i, \ell \geq 1$,
 \begin{eqnarray*}
 \pp (X_{i} \leq X_{i-1} + \ell \mid X_{i-2} < X_{i-1}) 
  & \geq & \prod_{\substack{j<0 \\ k > \ell }}   \left(1- |j-k|^{-s} \right)\\
 & \geq & 1 -  \sum_{\substack{j<0 \\ k > \ell }}   |j-k|^{-s}.
 \end{eqnarray*}
We have
\begin{eqnarray*}
\sum_{\substack{j<0 \\ k > \ell }}  |j-k|^{-s}   & = & \sum \limits_{j=1}^{ \infty } \sum\limits_{k=\ell+1}^{\infty} (k + j )^{-s} \\
& \leq & \frac{1}{s-1} \sum \limits_{j=1}^{ \infty } (j + \ell )^{1-s} \\
& \leq & \ell^{2-s}/(s-2).
\end{eqnarray*}
Therefore,
\begin{align*}
\pp( \varepsilon_i  > \ell \mid T \geq i ) \leq  \ell^{2-s}/(s-2), 
\end{align*}
which proves  (iii).
\end{proof}
Since the definition of $\lambda_c$ is independent of the starting vertex, we can assume that the initially infected vertex is a cut-point. 

It will be convenient to assume that $0$ is a cut-point. Suppose that conditioned on $0$ being a cut-point and infected at the beginning, we can prove that $\lambda_c >0$. Since the distribution is invariant under translations, we have $\lambda_c >0$ for the contact process starting from any cut point.

Hence, from now on we condition on the event $0$ is a cut-point. Set $K_0=0$,  for $i\geq 1$, we call $K_i$ (resp. $K_{-i}$) the $i^{th}$ cut point from the right (resp. left) of $0$. By Lemma \ref{lgd} (ii), the  graphs induced in the intervals $[K_i, K_{i+1})$ are i.i.d.  Therefore, $G_s$ is isomorphic to the graph $\tilde{G}_s$ obtained by gluing an i.i.d. sequence of graphs with distribution of the graph $[0,K_1)$. We have to prove that the contact process on $\tilde{G}_s$ exhibits a non-trivial phase transition.
\subsection{Cumulatively merged partition}
We recall here the  definitions introduced in \cite{MS}. Given a locally finite graph $G=(V,E)$, an expansion exponent $\alpha \geq 1$, and a sequence of non-negative weights defined on the vertices 
\[(r(x), x \in V) \in [0, \infty)^V,\]  
a partition $\kC$ of the vertex set $V$ is said to be $(r, \alpha)$-admissible if it satisfies 
\[\forall \, C, C' \in \kC, \quad C \neq C' \qquad \Longrightarrow \qquad d(C,C') > \min \{r(C),r(C')\}^{\alpha},\]
with 
\[r(C) = \sum \limits_{x \in C} r(x).\]
We call \textit{cumulatively merged partition} (CMP) of the graph G with respect to $r$ and $\alpha$ the finest $(r, \alpha)$- admissible partition and denote it by $\mathscr{C}(G,r, \alpha)$. It is the intersection of all $(r, \alpha)$-admissible partitions of the graph, where the intersection is defined as follows: for any sequence of partitions $(\kC_i)_{i \in I}$, 
\begin{align*}
x\sim y \textrm{ in } \cap_{i\in I}\kC_i \quad \textrm{ if } \quad x \sim y \textrm{ in $\kC_i$ for all } i \in I.
\end{align*}
As for Bernoulli percolation on $\mathbb{Z}^d$, the question we are interested in is  the existence of an infinite cluster (here an infinite partition).  For the CMP on $\mathbb{Z}^d$ with i.i.d. weights, we have the following result.
\begin{proposition} \cite[Proposition 3.7]{MS} \label{pms}
For any $\alpha \geq 1$, there exists a positive constant $\beta_c=\beta_c(\alpha)$, such that for any  positive random variable $Z$ satisfying $\E(Z^{\gamma}) \leq 1$ with $\gamma =(4 \alpha d)^2$ and any  $\beta< \beta_c$, almost surely $\sC(\mathbb{Z}^d, \beta Z,\alpha)$-the CMP  on $\mathbb{Z}^d$ with expansion exponent $\alpha$ and i.i.d. weights distributed as $\beta Z$-has  no infinite cluster. 
\end{proposition}
We note that in \cite[Proposition 3.7]{MS}, the authors only assume that $\en(Z^{\gamma}) < \infty$ and they do not   precise the dependence of   $\beta_c$ with $\E(Z^{\gamma})$.  However, we can deduce from their proof a lower bound on $\beta_c$ depending only on $\E(Z^{\gamma})$ (and only on $\alpha, \gamma, d$ if we suppose $\en(Z^{\gamma}) \leq 1$), see  Appendix for more details. Finally, our $\beta_c(\alpha)$ is a lower bound of the critical parameter $\lambda_c(\alpha)$ introduced by M\'enard and Singh.

Using the notion of CMP, they give a sufficient condition on a graph $G$ ensuring that the critical value of the contact process is  positive. 
\begin{theorem} \cite[Theorem 4.1]{MS} \label{tms}
Let $G=(V,E)$ be a locally finite connected graph. Consider $\mathscr{C}(G,r_{\vartriangle},\alpha)$  the   CMP on G with  expansion exponent $\alpha$ and degree weights
$$r_{\vartriangle}(x)= \deg(x) 1(\deg(x) \geq \vartriangle).$$
Suppose that for some   $\alpha \geq5/2$ and  $\vartriangle \geq 0$, the partition $\mathscr{C}(G,r_{\vartriangle},\alpha)$ has no infinite cluster. Then 
\[\lambda_c(G)>0.\]
\end{theorem} 
Thanks to this result,  Theorem \ref{the} will follow from the following proposition.
\begin{proposition} \label{lch}
Fix $s > 102$.   There exists a positive constant $\vartriangle$, such that the partition $\sC(\tilde{G}_s, r_{\vartriangle},  5/2)$ has no infinite cluster a.s. 
\end{proposition}
\subsection{Proof of Proposition \ref{lch}}
Let $\sC_1$ and $\sC_2$ be two CMPs. We write $\sC_1 \preceq \sC_2$, if there is a coupling such that $\sC_1$ has an infinite cluster only if $\sC_2$ has an infinite cluster.   
\begin{lemma} \label{lcp} We have
\begin{align} \label{bb}
\sC(\tilde{G}_s, r_{\vartriangle}, 5/2) \preceq \sC(\mathbb{Z}, Z_{\vartriangle}, 5/2),
\end{align}
with
\[Z_{\vartriangle} = \sum \limits_{0 \leq x < K_1} \deg(x) 1(\deg(x) \geq \vartriangle).\]
\end{lemma}
\begin{proof}
For $i \in \Z$, we define
\[Z_i= \sum \limits_{K_i \leq x < K_{i+1}} \deg(x) 1(\deg(x) \geq \vartriangle). \]
Then $(Z_i)_{i \in \Z}$ is a sequence of i.i.d. random variables with the same distribution as $Z_{\vartriangle}$, since the graph $\tilde{G}_s$ is composed of i.i.d. subgraphs  $[K_i,K_{i+1})$.   Therefore, $\sC(\mathbb{Z}, (Z_i), 5/2)$ has the same law as $\sC(\mathbb{Z}, Z_{\vartriangle}, 5/2)$. Thus to prove Lemma \ref{lcp}, it remains to show that 
\begin{equation} \label{gzi}
\sC(\tilde{G}_s, r_{\vartriangle}, 5/2)  \preceq \sC(\mathbb{Z}, (Z_i), 5/2).
\end{equation}
For any subset $A$ of the vertices of $\tilde{G}_s$, we define its projection
\[p(A)=\{i \in \Z:  A \cap [K_i,K_{i+1}) \neq \varnothing\}.\]
Since all  intervals $[K_i,K_{i+1})$ have finite mean, if $|A|= \infty$ then $|p(A)|=\infty$. Therefore, to prove  \eqref{gzi}, it suffices to  show that 
\begin{equation} \label{xsy}
x \sim y  \textrm{ in } \sC(\tilde{G}_s, r_{\vartriangle}, 5/2) \quad  \textrm{ implies } \quad p(x) \sim p(y)  \textrm{ in } \sC(\mathbb{Z}, (Z_i), 5/2).
\end{equation}
We prove \eqref{xsy} by contradiction. Suppose that there exist $x_0$ and $y_0$ such that $x_0 \sim y_0  $  in  $\sC(\tilde{G}_s, r_{\vartriangle}, 5/2)$ and $ p(x_0) \not \sim p(y_0) $   in  $ \sC(\mathbb{Z}, (Z_i), 5/2)$. Then by definition there exists $\kC$, a $((Z_i), 5/2)$-admissible partition of $\Z$, such that $p(x_0) \not \sim p(y_0)$ in $\kC$.

\vspace{0.2 cm}
 We define a partition $\tilde{\kC}$ of $\tilde{G}_s$ as follows:
\[x \sim y \textrm{ in } \tilde{\kC} \quad \textrm{ if and only if } \quad p(x) \sim p(y) \textrm{ in } \kC.\]
In other words, an element in $\tilde{\kC}$ is  $\cup_{i \in C} [K_i,K_{i+1})$ with $C$ a set  in $\kC$. We now claim that $\tilde{\kC}$ is $(r_{\vartriangle}, 5/2)$-admissible. Indeed, let $\tilde{C}$ and $\tilde{C}'$ be two different sets in $\tilde{\kC}$. Then by the definition of $\tilde{\kC}$, we have  $p(\tilde{C})$ and $p(\tilde{C}')$ are two different sets in $\kC$ and 
\[Z(p(\tilde{C})):= \sum_{i \in p(\tilde{C})} Z_i = \sum_{x \in \tilde{C}} \deg(x) 1(\deg(x) \geq \vartriangle)=r_{\vartriangle}(\tilde{C}).\]
 Moreover, since these intervals $[K_i,K_{i+1})$ are disjoint,  
\[d(\tilde{C},\tilde{C}') \geq d(p(\tilde{C}),p(\tilde{C}')). \]
On the other hand, as $\kC$ is $((Z_i),5/2)$-admissible,
\[d(p(\tilde{C}),p(\tilde{C}')) > \min \{Z(p(\tilde{C})),Z(p(\tilde{C}'))\}^{5/2}.\]
It follows from the last three inequalities that 
\[d(\tilde{C},\tilde{C}') > \min \{r_{\vartriangle}(\tilde{C}),r_{\vartriangle}(\tilde{C}')\}^{5/2}, \]
which implies that $\tilde{\kC}$ is $(r_{\vartriangle}, 5/2)$-admissible.

\vspace{0.2 cm}
Let $C_0$ and  $C_0'$  be the two sets  in the partition $\kC$ containing $p(x_0)$ and $p(y_0)$ respectively. Then by assumption $C_0 \neq C_0'$.  We define 
\[\tilde{C}_0 = \bigcup_{i \in C_0} [K_i,K_{i+1}) \qquad \textrm{ and } \qquad \tilde{C}_0' = \bigcup_{i \in C_0'} [K_i,K_{i+1}). \]
Then  both $\tilde{C}_0$ and $  \tilde{C}'_0 $ are in $ \tilde{\kC}$, and  $\tilde{C}_0 \neq \tilde{C}'_0$. Moreover $\tilde{C}_0 $ contains $x_0$ and $\tilde{C}'_0 $ contains $y_0$. Hence $x_0 \not \sim y_0$ in  $\tilde{\kC}$ which is a $(r_{\vartriangle}, 5/2)$-admissible partition. Therefore, $x_0 \not \sim y_0$ in $\sC(\tilde{G}_s, r_{\vartriangle}, 5/2)$, which leads to a contradiction. Thus \eqref{xsy} has been proved.
\end{proof}
\noindent We now apply Proposition \ref{pms} and Lemma \ref{lcp} to prove Proposition \ref{lch}. To do that, we fix a positive constant $\beta < \beta_c (5/2)$ with $\beta_c(5/2)$ as in Proposition \ref{pms} with $d=1$ and rewrite 
\[Z_{\vartriangle} = \beta \frac{Z_{\vartriangle}}{\beta}.\]
If we can show that there is $\vartriangle = \vartriangle(\beta,s)$, such that
\begin{align} \label{eq}
\E\left( \left( \frac{Z_{\vartriangle}}{\beta}\right)^{100} \right) \leq 1,
\end{align}
 then Proposition \ref{pms} implies that a.s.  $\sC(\mathbb{Z}, Z_{\vartriangle}, 5/2)$ has  no infinite cluster. Therefore, by Lemma \ref{lcp},  there is no infinite cluster in  $\sC(\tilde{G}_s, r_{\vartriangle}, 5/2)$ and thus Proposition \ref{lch} follows. Now it remains to prove \eqref{eq}.

It follows from Proposition \ref{pgd} (i)  that 
\begin{align} \label{ng1}
\E(K_1^{100}) = \E(D^{100}) = \E \left( \left(\sum \limits_{i=0}^{T } \varepsilon _i \right)^{100} \right), 
\end{align}
where $T$ and $(\varepsilon_i)$ are as in Proposition \ref{pgd}. 

Applying the inequality   $(x_1+\ldots+ x_n)^{100} \leq n^{99}(x_1^{100}+ \ldots+ x_n^{100})$  for any $n\in \mathbb{N}$ and $x_1,\ldots,x_n \in  \mathbb{R}$, we  get  
\begin{eqnarray} \label{bfk}
 \E \left( \left(\sum \limits_{i=0}^{T } \varepsilon _i \right)^{100} \right) & \leq &  \E\left[ (T+1)^{99}  \sum \limits_{i=0}^{T} \varepsilon_i^{100}  \right] \notag \\
& = & \sum \limits_{i=0}^{\infty}   \E\left[ (T+1)^{99}   \varepsilon_i^{100} 1(T \geq i)  \right].   
\end{eqnarray}
Let $p = 1 + (s-102)/200 >1$ and $q$ be its conjugate, i.e. $p^{-1}+q^{-1}=1$. Then applying  H\"older's inequality, we obtain
\begin{align} \label{ngn}
\E\left[ (T+1)^{99}   \varepsilon_i^{100} 1(T \geq i) \right] \leq \E \left( (T+1)^{99q} \right)^{1/q} \E \left(\varepsilon_i^{100p}1(T\geq i) \right)^{1/p}.
\end{align}
On the other hand,  
\begin{equation} \label{gtn}
 \E \left(\varepsilon_i^{100p}1(T\geq i) \right) =  \E \left(\varepsilon_i^{100p} \mid T\geq i \right) \pp(T \geq i).
\end{equation}
Using Proposition \ref{pgd} (iii) we have for $i\geq 1$
\begin{eqnarray*}
 \E \left(\varepsilon_i^{100p} \mid T\geq i \right) & \leq & 100p\sum \limits_{\ell =0}^{\infty} \pp(\varepsilon_i > \ell \mid T \geq i) (\ell+1)^{100p-1} \\
 & \leq & 100p \left[1+  \sum \limits_{\ell \geq 1} \ell^{2-s} (1+ \ell)^{100p-1}/(s-2) \right]\\
 &\leq & C_1= C_1(s) < \infty,
\end{eqnarray*}
since by definition
\[2-s+ 100p-1 = -1 -(s-102)/2< -1.\]
Hence for all  $i\geq 1$
\begin{align}\label{ng2}
\E \left(\varepsilon_i^{100p}1(T\geq i) \right) \leq C_1\pp(T\geq i).
\end{align}
It follows from \eqref{ng1}, \eqref{bfk}, \eqref{ngn} and \eqref{ng2} that 
\begin{eqnarray} \label{huz}
 \E(K_1^{100}) & \leq & \E\left[(T+1)^{99q}\right]^{1/q} \left[1 +   \sum_{i=1}^{\infty} \left( C_1 \pp(T\geq i)\right)^{1/p} \right] \notag \\
& = & M < \infty,
\end{eqnarray}
since $T$ is stochastically dominated by a geometric random variable. 

For any $j \in \mathbb{Z}$ and any  interval $I$, we denote by $\deg_{I}(j)$ the number of  neighbors of $j$ in $I$ when we consider the original graph (without conditioning on $0$ being a cut-point).

Now for any non decreasing sequence $(x_k)_{k\geq 1}$ with $x_1 \geq 1$, conditionally on  $\varepsilon_1 =x_1 -1, \varepsilon_2 =x_2 -x_1, \ldots$, we have for all $j \in (x_{k-1}, x_k)$, 
\begin{align*}
 \deg(j) \prec 1+ \deg_{[x_{k-2} , x_{k+1})}(j),
\end{align*}
where $\prec$ means stochastic domination. 

Indeed, the conditioning implies that $j$ is only connected   to  vertices in $[x_{k-2},x_{k+1}]$ and that there is a vertex in $[x_{k-1},x_k)$ connected to $x_{k+1}$.

Similarly, if $j=x_{k}$, it is only connected to  vertices in $[x_{k-2},x_{k+2}]$. Moreover,  $j$ is connected  to at least one vertex in $[x_{k-2},x_{k-1})$ and there is a vertex in $[x_k, x_{k+1})$ connected  to $x_{k+2}$. Therefore, 
\begin{align*}
\deg(x_{k}) \prec 2+ \deg_{[x_{k-2} , x_{k+2})}(x_k).
\end{align*}
In conclusion, conditionally on  $j \in [0,K_1)$,
\begin{align*}
\deg(j) \prec 2 + Y,
\end{align*}
where 
\[ Y= \deg_{(- \infty, + \infty)}(j).\]
Hence,
\begin{align} \label{bd1}
\E \left(\deg(j)^{100} 1 (\deg(j) \geq \vartriangle) \mid j \in [0, K_1) \right)  \leq \E \left((2+Y)^{100} 1 (Y\geq \vartriangle -2) \right).
\end{align}
On the other hand,
\begin{eqnarray*}
\pp(Y= k) &= & \pp( \deg_{(-\infty,+\infty)}(0) = k)  \\
&\leq & \pp(\deg_{(-\infty,+\infty)}(0) \geq k) \\
 &\leq &  \sum \limits_{i_1<i_2< \ldots<i_k} |i_1|^{-s} |i_2|^{-s} \ldots |i_k|^{-s} \\
 &\leq &\frac{1}{k!} \sum \limits_{i_1,i_2, \ldots,i_k} |i_1|^{-s} |i_2|^{-s} \ldots |i_k|^{-s} \\
 &= & \frac{1}{k!} \left( 2 \sum \limits_{i \geq 1}  i^{-s} \right) ^k   = \frac{C^k}{k!}, 
\end{eqnarray*}
with $C= 2 \sum_{i \geq 1} i^{-s}$. Therefore,
\begin{eqnarray} \label{bd2}
\E\left( (2   + Y)^{100} 1 (Y\geq \vartriangle -2) \right) & \leq & \sum \limits_{k \geq \vartriangle -2}  \frac{C^k (k+2)^{100} }{k!}  \notag\\
& := & f(\vartriangle).
\end{eqnarray}
It follows from \eqref{huz}, \eqref{bd1} and \eqref{bd2} that 
\begin{eqnarray*}
\E(Z_{\vartriangle}^{100}) & = & \E \left[ \left( \sum \limits_{0 \leq j < K_1} \deg(j) 1(\deg(j) \geq \vartriangle)\right)^{100}\right] \\
& \leq & \E \left[ K_1 ^{99} \sum \limits_{0 \leq j < K_1} \deg(j)^{100} 1(\deg(j) \geq \vartriangle)\right] \\
& \leq & \E(K_1 ^{100}) f(\vartriangle) \\
& \leq & M f(\vartriangle).
\end{eqnarray*}
Since $f(\vartriangle) \rightarrow 0$ as $\vartriangle \rightarrow \infty$, there exists $\vartriangle_0   \in (0,  \infty)$, such that  $M f(\vartriangle_0) \leq \beta^{100}$ and thus \eqref{eq} is satisfied. \hfill $\square$

\section*{Appendix: a lower bound on $\beta_c$}
In \cite{MS},   Proposition 3.7  (our Proposition \ref{pms}) follows from Lemmas 3.9, 3.10, 3.11 and a conclusion argument. Let us find in their proof a lower  bound on $\beta_c$. 

At first, they define a constant $c=2\alpha d +1$ and some sequences 
\begin{align*}
L_n = 2^{c^n} \quad \textrm{and} \quad R_n= L_1 \ldots L_n \quad \textrm{and} \quad \varepsilon_n = 2^{-2dc^{n+1}}.
\end{align*}
In Lemma 3.9, the authors do  not use any information on  $Z$ and $\beta$. They  set a constant $k_0= [2^{d+1}(c+1)]$.

In Lemma 3.10, they suppose that $\beta \leq 1$ and the information concerning $Z$  is as follows.  There exists $n_0$, such that for all $n \geq n_0$, we have
\begin{align*}
2^{d} \E(Z^{\gamma})L_{n+1}^{-\mu} \leq 1/2,
\end{align*}
with 
\[\mu= \frac{\gamma -1}{2 \alpha} - 3d - 4 \alpha d^2 >0. \]
In fact, under the assumption $\en(Z^{\gamma}) \leq 1$, we can take 
\begin{align} \label{n0}
n_0= \left[ \frac{\log \left( \frac{d+1}{\mu }\right)}{\log c} \right].
\end{align}
In Lemma 3.10, they also assume that $\beta \leq 1$ and  define a constant $n_1$, such that $n_1 \geq n_0$ and for all $n \geq n_1$
\[3k_0^{\alpha +1} L_{n+1} \leq \frac{R_{n+1}}{20} ,\]
or equivalently, 
\begin{align} \label{n1}
60 k_0^{\alpha +1}  \leq R_n.
\end{align}
In the conclusion leading to the proof of  \cite[Proposition 3.7]{MS}, a lower bound on $\beta_c$ is implicit. Indeed, with Lemmas 3.9, 3.10, 3.11 in hand, the authors only require that 
  \begin{align} \label{ern}
  \pn(\kE(R_{n_1})) \geq 1 - \varepsilon_{n_1},
  \end{align}
 where for any $N \geq 1$ 
\[\kE(N)= \{\textrm{there exits a stable set $S$ such that } \llbracket N/5, 4N/5 \rrbracket ^d \subset S \subset \llbracket 1, N \rrbracket ^d  \}.\]
 We do not recall the definition of stable sets here. However, we notice that by the first part of Proposition 2.5  and Corollary 2.13 in \cite{MS}, the event $\kE(N)$ occurs when the weights of all  vertices in $\llbracket 1, N \rrbracket ^d$ are less than $1/2$. Therefore
\begin{eqnarray*}
\pn(\kE(N)) &\geq& \pn \left( r(x) \leq 1/2 \textrm{ for all } x \in \llbracket 1, N \rrbracket ^d \right) \\
& = & \pn(\beta Z \leq 1/2)^{N^d}  \\
& =&  \left( 1 - \pn(\beta Z > 1/2) \right) ^{N^d} \\
& = & \left( 1 - \pn(Z^{\gamma} > (2\beta)^{-\gamma}) \right) ^{N^d} \\
& \geq &  \left( 1 - (2\beta)^{\gamma}\en(Z^{\gamma}) \right) ^{N^d}.
\end{eqnarray*} 
Hence \eqref{ern} is satisfied if
   \[\left( 1 - (2\beta)^{\gamma}\en(Z^{\gamma}) \right) ^{R_{n_1}^d} \geq (1- \varepsilon_{n_1}),\]
   or equivalently 
 \[(2\beta)^{\gamma}\E(Z^{\gamma}) \leq 1-(1-\varepsilon_{n_1})^{R_{n_1}^{-d}}.\]
Hence, under the assumption  $\E(Z^{\gamma})  \leq 1$,   we can take 
\[\beta_c = \frac{1}{2} \left(1-(1-\varepsilon_{n_1})^{R_{n_1}^{-d}} \right)^{1/ \gamma},\]
with $n_1$ as in \eqref{n1}.

\ACKNO{
 I am deeply grateful to my  advisor Bruno Schapira for his help and many suggestions during the preparation of this work. I would like also to thank the referee for a careful reading this paper as well as many valuable comments.}

\end{document}